\documentclass[onefignum,onetabnum]{siamart171218}
\usepackage{subcaption,tikz-cd}
\usepackage{lipsum}
\usepackage{amsfonts}
\usepackage{graphicx}
\usepackage{epstopdf}
\usepackage{algorithmic}
\usepackage{color}
\usepackage[normalem]{ulem}
\ifpdf
  \DeclareGraphicsExtensions{.eps,.pdf,.png,.jpg}
\else
  \DeclareGraphicsExtensions{.eps}
\fi

\newsiamremark{remark}{Remark}
\newsiamremark{hypothesis}{Hypothesis}
\crefname{hypothesis}{Hypothesis}{Hypotheses}
\newsiamthm{fact}{Fact}

\headers{Near-Circulant Splitting}{E. K. Ryu, S. Ko, and J.-H. Won}

\title{Splitting with Near-Circulant Linear Systems:\\
Applications to
Total Variation CT and PET\thanks{Submitted to the editors October 30, 2018.
\funding{EKR was supported in part by NSF grant DMS-1720237 and ONR grant
N000141712162. SK and JHW were supported by the National Research Foundation of Korea
(NRF) grant funded by the Korea government (MSIT) (No. NRF-2019R1A2C1007126).
}}}

\author{Ernest K. Ryu\thanks{Department of Mathematics, UCLA
  (\email{eryu@math.ucla.edu}).}
\and Seyoon Ko\thanks{Department of Statistics, Seoul National University
  (\email{syko0507@snu.ac.kr}, \email{wonj@stats.snu.ac.kr}).}
\and Joong-Ho Won\footnotemark[3]}

\usepackage{amsopn}
\newcommand{\reals}{\mathbb{R}}
\newcommand{\diag}{\mathrm{diag}}
\DeclareMathOperator*{\argmin}{arg\,min}
\newcommand{\prox}{{\mathbf{prox}}}

\ifpdf
\hypersetup{
  pdftitle={Near-Circulant Splitting},
  pdfauthor={Ernest K. Ryu, Seyoon Ko, Joong-Ho Won}
}
\fi

\begin{document}

\maketitle

% REQUIRED
\begin{abstract}
Many imaging problems, such as total variation reconstruction of X-ray computed tomography (CT) and positron-emission tomography (PET), are solved via a convex optimization problem with near-circulant, but not actually circulant, linear systems. The popular methods to solve these problems, alternating direction method of multipliers (ADMM) and primal-dual hybrid gradient (PDHG), do not directly utilize this structure. Consequently, ADMM requires a costly matrix inversion as a subroutine, and PDHG takes too many iterations to converge. In this paper, we present near-circulant splitting (NCS), a novel splitting method that leverages the near-circulant structure. We show that NCS can converge with an iteration count close to that of ADMM, while paying a computational cost per iteration close to that of PDHG. Through experiments on a CUDA GPU, we empirically validate the theory and demonstrate that NCS can effectively utilize the parallel computing capabilities of CUDA.
\end{abstract}

% REQUIRED

% REQUIRED

\section{Introduction}
Many imaging problems are solved via the optimization problem
\begin{equation}
\begin{array}{ll}
\underset{\substack{x\in \reals^n
%\\z\in \reals^m
}}{
\mbox{minimize}}&g(Ax-b),
%\\\mbox{subject to}& =z
\end{array}
\label{eq:RCP}
\end{equation}
where $g$ is a convex function,
$A\in \reals^{m\times n}$,
$b\in \reals^m$,
%is a matrix containing something like the Radon transform and the finite difference operator,
%$b\in \reals^m$ contains measurements,
and the variable $x\in \reals^n$ represents the (vectorized) image to reconstruct or restore.
%We call \eqref{eq:RCP} a \emph{range constrained program}, since it is equivalent to
%\[
%\begin{array}{ll}
%\underset{z\in \reals^m}{
%\mbox{minimize}}&g(z)\\
%\mbox{subject to}& z\in \mathcal{R}(A)-b
%\end{array}
%\]

One popular method to solve \eqref{eq:RCP} is alternating direction method of multipliers (ADMM) or equivalently Douglas--Rachford splitting (DRS)
\begin{align}
x^{k+1}&=x^k- \alpha^{-1}( A^TA)^{+}A^Tu^k
\tag{ADMM}
\label{eq:DRS/ADMM}
\\
u^{k+1}&=
\prox_{\alpha g^*}(u^{k}+\alpha (A(2x^{k+1}-x^k)- b))\nonumber
\end{align}
which converges for any $\alpha>0$
\cite{gabay1976,glowinski1975,peaceman1955,douglas1956,lions1979, kellogg1969,split_bregman2009,ChCaCrNoPo10,Ramani2012}.
 (This is not the usual form of ADMM. We performed a few changes of variables,
 as shown in Section~\ref{s:appendix}.)
Here,  $(A^TA)^{+}$ denotes the pseudoinverse of $A^TA$.\
When we can run ADMM,  empirical evidence supports that it converges with relatively few iterations.
However, we often cannot run ADMM, as $(A^TA)^+$ is too expensive to (directly) compute when, say,  $n\ge 256^2$.
If $A^TA$ is circulant, we can compute $(A^TA)^+$  efficiently with the fast Fourier transform (FFT),
but this is not the case in the applications we consider.
The cost of even a single iteration of ADMM is prohibitively expensive. 
While it is possible to solve the linear system iteratively with an inner loop, doing so raises the question of when to terminate the inner loop and under what condition can we theoretically ensure convergence of the approximate ADMM.

% and when the number of iterations should be considered.

% but they have the problem of how to tuning the numnber of 
% but these often suffer from lack of accuracy needed for ADMM to converge.
% But the lack of accuracy, the number of iterations and so on.

Another popular method to solve \eqref{eq:RCP} is primal-dual hybrid gradient (PDHG), also known as the Chambolle--Pock method,
\begin{align}
x^{k+1}&=x^k-\gamma^{-1} A^Tu^k
\tag{PDHG}
\label{eq:PDHG}
\\
u^{k+1}&=
\prox_{\alpha g^*}(u^k+\alpha (A(2x^{k+1}-x^k)- b))\nonumber
\end{align}
which converges for $\gamma /\alpha\ge \lambda_\mathrm{max}(A^TA)$ \cite{zhu08,pock2009,esser2010,ChCaCrNoPo10,ChaPoc11,sidky2012}.
The application of $(A^TA)^+$ is the only difference between ADMM and PDHG.
While PDHG has a low cost per iteration, the method often takes too many iterations, and empirically speaking the total cost, (cost per iteration) $\times$ (\# of iterations), is large.

In the imaging applications we consider,
the matrix $A^TA$ is not circulant, but near-circulant.
Because $A^TA$ is big and not circulant, we cannot use ADMM. 
%We can use PDHG, but it takes too many iterations.
On the other hand, PDHG may take prohibitively many iterations.
In this paper, we present near-circulant splitting (NCS), a splitting method that leverages the near-circulant structure of $A^TA$.
The method converges with an iteration count close to that of ADMM,
while paying a cost per iteration close to that of PDHG.

\subsection{Near-circulant matrices}
We say $A^TA$ is \emph{near-circulant} 
if there is a circulant matrix $C$ such that  $C\approx A^TA$.
%{\color{red}More precisely, if $C$ and $A^TA$ are different discretization of a linear operator $\mathcal{A}$ and $C$ is circulant, then $A^TA$ is near-circulant.}
Near-circulant matrices arise in many imaging applications.
Consider, for example, %total variation CT imaging:
reconstruction of X-ray computed tomography (CT) with 2D parallel beam geometry under total variation penalty \cite{jain1989fundamentals,rudin1992nonlinear}:
\[
\begin{array}{ll}
\mbox{minimize}&
g(y,z)\\
\mbox{subject to}&
\begin{bmatrix}
R\\D
\end{bmatrix}
x-
\begin{bmatrix}
b\\0
\end{bmatrix}=
\begin{bmatrix}
y\\z
\end{bmatrix}
\end{array}
\]
where $g(y,z)=(1/2)\|y\|^2+\lambda \|z\|_1$,
$R$ is the discrete Radon transform, and $D$ is the finite difference operator.
Sections \ref{s:prelim} and \ref{s:appl} provide further details.
With
\[
A=\begin{bmatrix}
R\\ D
\end{bmatrix},
\qquad
A^TA=
R^TR+D^TD,
\]
we have an instance of \eqref{eq:RCP}.
If $R^TR$ and $D^TD$ are near-circulant, then so is $A^TA$.
%The square matrix $A^TA$ is large, and its pseudoinverse $(A^TA)^+$ is in general difficult to compute.

%The square matrix $R^TR$ is large, and its pseudoinverse $(R^TR)^{+}$ is in general difficult to compute.
%In many imaging applications, we have a linear operator $\mathcal{R}$ on a continuous image
%and its discretization $R$, such as the continuous and discretized Radon transforms.

%In particular, we have
%\[
%\mathcal{R} \xrightarrow{\mathrm{discretize}} R
%\]

%In our applications, $\mathcal{R}$ is the Radon transform and $R$ is its discretization.
Let $\mathcal{R}$ denote the (continuous) Radon transform.
Then $R$ is a discretization of $\mathcal{R}$.
Write $\mathcal{R}^*$ for the adjoint of $\mathcal{R}$, which represents backprojection.
Then $R^T$ is a discretization of $\mathcal{R}^*$
and $R^TR$ is a discretization of $\mathcal{R}^*\mathcal{R}$,
a linear spatially invariant operator \cite{deans1983radon}.
%\sout{This makes an intuitive sense since in imaging applications we want a translated input object to produce the same, but translated, output.}
Since the Fourier transform operator $\mathcal{F}$ diagonalizes linear spatially invariant operators, 
\[
%\mathcal{R}^*\mathcal{R}f =\mathcal{F}^{-1}\left[h(\omega)\hat{f}(\omega)\right]
\mathcal{R}^*\mathcal{R}f =\mathcal{F}^{-1}(\hat{k}\hat{f})
\]
for some $\hat{k}:\mathbb{R}^2\to\mathbb{C}$, where $\hat{f}=\mathcal{F}(f)$.
We can discretize in the Fourier domain and obtain an alternative discretization $F^{-1}\diag (h) F$ of $\mathcal{R}^*\mathcal{R}$ for some $h\in \reals^n$, where $F$ denotes the discrete Fourier transform matrix.
%and $n$ is the grid size.
%This is the basis of the convolution backprojection algorithm for CT \cite{jain1989fundamentals}.

We now have two discretizations of $\mathcal{R}^*\mathcal{R}$ that are not necessarily equal:
\[
F^{-1}\mathrm{diag}(h)F\ne R^TR.
\]
To put it differently,
discretizing $\mathcal{R}$ into $R$ and then forming $R^TR$
is not the same as forming $\mathcal{R}^*\mathcal{R}$ and then discretizing.
See Figure~\ref{fig:non-commute}.
However, the two are approximately equal,
since both approximate the same continuous linear operator $\mathcal{R}^*\mathcal{R}$.
Therefore $R^TR$ is near-circulant.

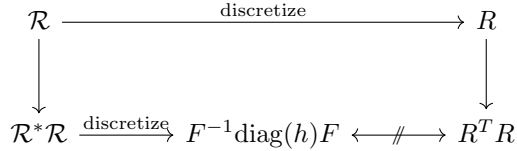
\begin{figure}[h]
\begin{center}
\begin{tikzcd}[sep=large]
\mathcal{R} \arrow{d}[left]{} 
\arrow{rr}{\mathrm{discretize}}
&
& R \arrow{d}{} \\
\mathcal{R}^*\mathcal{R} 
\arrow{r}{\mathrm{discretize}}
&  F^{-1}\diag(h) F
\arrow[r,leftrightarrow,"/\!\!/"{anchor=center,sloped}]
& R^TR
\end{tikzcd}
\end{center}
\caption{
Illustration of the two discretizations in consideration.}
\label{fig:non-commute}
\end{figure}

%As we discuss soon, $D^TD$ is also near-circulant, and we can write
The matrix $D^TD$ is a discretization of the Laplacian operator $\nabla^2$ under the Neumann boundary condition and is not circulant.
Discretizing $\nabla^2$ under the periodic boundary condition gives us a circulant discretization $\tilde{L}$, and
$\tilde{L} \approx D^TD$.
%F^{-1}\diag (g)F\approx D^TD
See Section \ref{s:prelim} for details.

Write $\tilde{L}=F^{-1}\diag (d)F$
for some $d\in \reals^n$.
So we can use 
\[
F^{-1}(\diag(h+d))F\approx A^TA
\]
as a circulant approximation to $A^TA$.
We can efficiently compute the pseudoinverse of this circulant approximation using the FFT.
%with small $\gamma$ as an approximation to $A^TA$ that is easily invertible.
%
%\[
%F^{-1}(\gamma I+\diag(h+g))^{-1}F\approx (A^TA)^+
%\]
%as a computationally efficient approximation to $(A^TA)^+$.

%However, the discretization of $\mathcal{R}^*\mathcal{R}$ is only approximately equal to $R^TR$.
%\[
%\mathcal{F}^{-1}\mathrm{diag}(H)\mathcal{F}
%\ne R^TR
%\]
%\[
%\mathcal{R}^*\mathcal{R}\xrightarrow{\mathrm{discretize}} \mathcal{F}^{-1}\mathrm{diag}(H)\mathcal{F}
%\]

So $A^TA$ is near-circulant, but why is $A^TA$ not actually circulant?
The first of the two reasons is boundary conditions.
The discretized operator, which represents a physical imaging setup, works on a bounded domain.
%, not on the full infinite 2-D or 3-D space.
Usually, periodic boundary conditions are not used since they would not represent the actual imaging setup. Without periodic boundary conditions, $A^TA$ fails to be circulant.
However, this reason is less interesting since zero-padding provides an easy computational fix for this issue \cite{Almeida2013}.
See \cite[Section~2.2]{cai_deblurring2016} for a detailed discussion on boundary conditions in imaging.

The more interesting reason is that the discretization does not preserve spatial invariance.
The applications of $\mathcal{R}^*\mathcal{R}$ or $R^TR$ entail a Cartesian to polar to Cartesian change of coordinates.
The continuous $\mathcal{R}^*\mathcal{R}$ is spatially invariant.
However, this change of coordinates breaks spatial invariance in the discretization $R^TR$.
%To the best of our knowledge, there is no easy way to resolve this issue.

\subsection{The main method}
We present an algorithm that leverages the near-circulant structure of $A^TA$
%ADMM requires the actual pseudoinverse, while PDHG requires no inverse.
%Finally, we present the general method
\begin{align}
x^{k+1}&=x^k- M^{+}A^Tu^k
\tag{NCS}
\label{eq:NCS}
\\
u^{k+1}&=
\prox_{\alpha g^*}(u^{k}+\alpha (A(2x^{k+1}-x^k)- b))\nonumber
\end{align}
which converges for $M\succeq \alpha A^TA$.
We call this method \emph{near-circulant splitting (NCS)}.
For the best performance, we choose
$M=\gamma I+\alpha C$, where $C$ is a circulant approximation to $ A^TA$ and $\gamma >0$ is small.

Since $M$ is circulant, we can compute $M^+$ efficiently via the FFT.
This makes the cost per iteration of NCS close to that of PDHG.
As we discuss in Section~\ref{s:theory},
the number of necessary iterations of NCS is close to that of ADMM when 
$M\approx \alpha A^TA$.

\emph{Remark.}
NCS (and ADMM) requires the computation of $M^+$, which by itself may be  numerically unstable.
However, the $u$-iterates depend on $M^+A^Tu^k$ through $AM^+A^Tu^k$,
so the stability of NCS requires the computation of $M^+$ to result in $AM^+A^T$ being stable, which holds in practice.
%
%
%However, most implementations of ADMM are, although slow, mostly stable.
%This is because the iteration depends on $(A^TA)^+A^Tu^k$ through $A(A^TA)^+A^Tu^k$.
%Therefore, stability of ADMM requires the computation of $(A^TA)^+$ to result in $A(A^TA)^+A^T$ being stable, which is practical.

%\emph{Remark.}
%ADMM requires the computation of $(A^TA)^+$, which is not just expensive but also numerically unstable.
%However, most implementations of ADMM are, although slow, mostly stable.
%This is because the iteration depends on $(A^TA)^+A^Tu^k$ through $A(A^TA)^+A^Tu^k$.
%Therefore, stability of ADMM requires the computation of $(A^TA)^+$ to result in $A(A^TA)^+A^T$ being stable, which is practical.

\subsection{Prior work and contributions}
ADMM and DRS were first presented in \cite{gabay1976,glowinski1975,peaceman1955,douglas1956,lions1979, kellogg1969},
%The equivalence of ADMM and DRS was first presented in \cite{gabay1983}.
and PDHG was first presented in \cite{zhu08,pock2009,esser2010,ChaPoc11}.
These methods have been successfully applied to medical imaging \cite{sidky2012,Ramani2012}.
%ADMM/DRS and PDHG have been applied to CT imaging \cite{sidky2012,Ramani2012}.
There has been a large body of work analyzing the convergence of ADMM/DRS and PDHG. 
The analysis of Section~\ref{s:theory} utilizes proof techniques developed in such past work \cite{bot2015,davis2015,Chambolle2016,connor2017,davis2017}.

Circulant and Toeplitz structures arise in a wide range of contexts and therefore have been studied extensively throughout the numerical linear algebra literature \cite{strang1986,cg_toeplitz_1996,toeplitz_linear_algebra2003,toeplitz_linear_algebra2013,xie2017}.
%``Straing's preconditioner'' \cite{strang1986,cg_toeplitz_1996}

One can interpret NCS as applying a preconditioner to PDHG.
The general idea of accelerating splitting methods with preconditioning is not new.
Pock and Chambolle studied the use of diagonal preconditioners applied to PDHG in \cite{pock2011}.
Bredies and Sun applied preconditioners inspired by PDEs to DRS,
with a focus on denoising and deblurring problems \cite{bredies2015,Bredies20152}.
These works present algorithms different from ours, have different convergence analyses, and do not utilize the near-circulant structure of the problem.
The modified linearized Bregman algorithm of Cai et al.\ \cite{cai_deblurring2016} applied to image deblurring
does utilize circulant preconditioners, but the algorithm and convergence analysis are different.

One can also interpret NCS as a partial ``linearization'' of ADMM. 
Deng and Yin briefly discussed this idea and suggested using it with the finite difference operator, but not an imaging operator, in a paragraph of Section 1.1 of \cite{Deng2016}.
They analyze the rate of linear convergence under stronger assumptions than ours, involving strong convexity and smoothness, but these assumptions are not met in the applications we consider as the total variation loss is neither strongly convex nor smooth.
In the general case, Deng and Yin establish convergence but do not analyze the rate.

The notion of near-circulant systems has been considered before.
Fessler and Booth in \cite{fessler1999} considered second order methods for differentiable optimization problems originating from medical imaging applications
and proposed approaches to overcome the fact that the Hessians are near-circulant and not circulant.
The setup in this work is different, since we do not assume differentiability, and
the viewpoint is different, since we utilize near-circulance as an advantageous property.

%{\color{red}As a related approach,}
Ramani and Fessler \cite{Ramani2012} used ADMM to solve \eqref{eq:RCP}
and used preconditioned conjugate gradient (PCG) with circulant approximations as preconditioners for the subproblem of computing the inverse.
(They use the terminology ``nearly shift invariant'' systems.)
In other words, they perform PCG as an inner loop.
In comparison, NCS is conceptually and practically simpler as it requires no consideration of when to terminate the inner loop.
For the applications we, as well as Ramani and Fessler, consider, a single iteration of NCS has essentially the same computational cost as a single iteration of PCG, and a single outer loop of Ramani and Fessler's approach requires multiple PCG iterations.

Zhao et al.\ \cite{zhao_hyperspectral2013} used the discrete Fourier and cosine transforms to respectively solve the ADMM subproblems with periodic and Neumann boundary conditions for the hyperspectral imaging problem with sparse and total variation regularization.
However, this approach of
using fast transforms does not apply to the linear system of our setup given by the discrete Radon transform.

Zhao et al.\ \cite{zhao_hyperspectral2013} used the discrete Fourier and cosine transforms to respectively solve the ADMM subproblems with periodic and Neumann boundary conditions for the hyperspectral imaging problem with sparse
and total variation regularization.
Lee et al.\ \cite{lee2017} considered a similar approach to large-scale regression problems in statistics.
However, these approaches of
using fast transforms do not apply to the linear system of our setup given by the discrete Radon transform.

O'Connor and Vandenberghe \cite{connor2014,connor_thesis} presented splitting methods for imaging problems with linear systems of the form of circulant plus sparse, which are similar to, but not the same as, the near-circulant linear systems considered in this work.
They presented several methods that utilize this problem structure and empirically compared their performances.
Their algorithms are different from ours, and they do not provide a theoretical justification of any observed speedup.

The key theoretical contribution of this work is the analysis of the rate of convergence of NCS.
% in a manner that depends on the choice of $M$.
The theory, presented in Section~\ref{s:theory}, informs us on how to choose $M$.
Bredies and Sun in \cite{bredies2015,Bredies20152} establish a $\mathcal{O}(1/k)$ rate on the ``ergodic'' iterates, but their work does not directly inform us of how to choose the preconditioner.

Another contribution of this work is identifying applications with near-circulant structures that NCS can leverage.
Section~\ref{s:appl} discusses medical imaging applications and provides circulant approximations
to their linear systems.
Section~\ref{s:exp} experimentally demonstrates that NCS provides a significant speedup,
which is consistent with the theory of Section~\ref{s:theory}.
In these applications, the dominant cost of NCS is the application of $A$ and $A^T$, which is embarrassingly parallel.
Section~\ref{s:exp} also demonstrates that NCS can be computationally accelerated with GPUs.

\section{Definitions and preliminaries}\label{s:prelim}
In this section, we set up the notation and review standard results.
For details, we refer the readers to 
standard references on linear algebra \cite{horn2012,golub2012,laub2005},
Radon transform and tomographic imaging \cite{deans1983radon,toft1996radon},
and convex analysis \cite{rockafellar1970,rockafellar1974,boyd2004,ryu2016,BauschkeCombettes2017_convex}.

%\paragraph{Linear algebra.}
%If a symmetric matrix $M\in\reals^{n\times n}$ 
%has only nonnegative eigenvalues, we say $M$ is positive semidefinite and write
Given a symmetric matrix $A\in\reals^{n\times n}$,
we write $A\succeq 0$ to denote that $A$ is positive semidefinite.
If $A,B\in\reals^{n\times n}$ are symmetric, write
$A\succeq  B$ if $A-B\succeq 0$.
For $A\succeq 0$, define the seminorm
\[
\|x\|_A=(x^TAx)^{1/2}.
\]
%The Cauchy-Schwartz inequality holds for seminorms, i.e.,
%\begin{equation}
%x^TMy\le \|x\|_M\|y\|_M.
%\label{eq:CS}
%\end{equation}
Given $A\in \reals^{n\times n}$, write $A^+$ for the Moore--Penrose pseudoinverse.
When $A$ is invertible, $A^+=A^{-1}$.
%A matrix $C\in \reals^{n\times n}$ is circulant if 
%$C_{ij}=c_{j-i}$ for all $i,j$ and for some $c_0,\dots,c_{n-1}$, where the subscripts of $c_{j-i}$ are interpreted periodically.
Given $h\in \reals^n$, write $\diag(h)\in \mathbb{R}^{n\times n}$ for the diagonal matrix
such that $(\diag(h))_{ii}=h_i$ for $i=1,\dots,n$.
Write $F\in \reals^{n\times n}$ for the discrete Fourier transform (DFT) matrix.

A square matrix $C\in \reals^{n\times n}$ is circulant if and only if
$C=F^{-1}\diag(h)F$
for some $h\in \reals^n$, i.e., $C$ is circulant if and only if the DFT basis diagonalizes $C$.
%Write $\diag(h)\in \reals^{n\times n}$.
This fact gives us the computationally efficient formula 
\[
C^+=F^{-1}\diag(h)^+F,
\]
where $\diag(h)^+$ is the diagonal matrix with
\[
(\diag(h))^+_{ii}=\left\{
\begin{array}{ll}
1/h_{i}&\text{ if }h_{i}\ne 0\\
0&\text{ otherwise}
\end{array}\right.
\]
for $i=1,\dots,n$.
Circulant matrices are the discrete analogue of spatially invariant linear operators.

%pseudoinverse
%\cite[p.\ 453]{horn2012}
%\cite[\S5.5.2]{golub2012}
%\cite[\S4]{laub2005}.

%\subsection{Circulant matrices}
%XXX references xX
%\cite[\S4.8.2]{golub2012}
%\cite[p.\ 100]{horn2012}

The one-dimensional (1D) finite difference matrix is
\[
D_n=
\begin{bmatrix}
1&-1&&&&\\
&1&-1&&&\\
& &\ddots&\ddots&\\
&&&1&-1&\\
&&&&1&-1
\end{bmatrix}
\in \reals^{(n-1)\times n},
\]
and $L_n=D^T_nD_n$, where
\[
	L_{n} = \begin{bmatrix} 1 &  -1    &        &        &    \\ 
			               -1 &   2    &   -1   &        &    \\
						      &  & \ddots & &    \\
							  &        &   -1   &    2   & -1 \\
							  &        &        &   -1   &  1   
			\end{bmatrix}
			\in \mathbb{R}^{n\times n}
	.
\]
As a circulant approximation, we use $\tilde{L}_n\approx L_n$, where
\[
\tilde{L}_{n} = \begin{bmatrix} 2 &  -1    &        &        &  -1  \\ 
  			                   -1 &   2    &   -1   &        &    \\
				    		      & & \ddots & &    \\
				    			  &        &   -1   &    2   & -1 \\
				    		   -1 &        &        &   -1   &  2   
			    \end{bmatrix}			\in \mathbb{R}^{n\times n}
	.
\]
In 2D, we use the 2D finite difference matrix $D_{2,n}$ and 
\[
D^T_{2,n}D_{2,n} \approx I_{n}\otimes \tilde{L}_{n} + \tilde{L}_{n}\otimes I_{n}.
\]
The matrices $L_n$ and $\tilde{L}_n$ respectively arise in the discrete Poisson equation
with Neumann and periodic boundary conditions.
We can think of $\tilde{L}_n$ as a circulant approximation to the Toeplitz system $L_n$.
Approximating a Toeplitz system with a circulant system is a technique that has been explored in other applications \cite{strang1986,chan1988optimal,cg_toeplitz_1996,oppenheim2009discrete}.
Since $\tilde{L}_{n} $ is circulant, the DFT basis diagonalizes it, 
and we have
\begin{gather}
\tilde{L}_{n}=F^{-1}\diag(d)F,
\label{eq:f-diff}
\\
d_{i+1}=4\sin^2(i\pi/n)\qquad \text{for }i=0,\dots,n-1.\nonumber
\end{gather}
An analogous result holds in the 2D case.

%\paragraph{Functional analysis.}
%{\color{red} \sout{Write $\mathcal{F}$ for the (continuous) Fourier transform and $\mathcal{R}$ for the (continuous) Radon transform. Write $\mathcal{R}^*$ for the adjoint of $\mathcal{R}$, which represents backprojection.Then}}
%\begin{equation}
%\mathcal{R}^*\mathcal{R}f=
%\mathcal{F}^{-1}
%\left[
%\frac{1}{\|\omega\|}
%\hat{f}(\omega)\right],
%\label{eq:radon-formula}
%\end{equation}
%where $\hat{f}=\mathcal{F}(f)$.
The 2D Radon transform $\mathcal{R}f$ of a function $f:\mathbb{R}^2\to\mathbb{R}$ is
\begin{equation}\label{eq:radon-def}
	[\mathcal{R}f](s,\theta) = \int_{-\infty}^\infty\int_{-\infty}^\infty f(u,v)\delta(u\cos\theta+v\sin\theta -s )\;du dv,
	\quad
	s \in (-\infty,\infty), ~\theta\in[0,\pi)
	,
\end{equation}
where $\delta$ denotes the Dirac delta function.
The adjoint $\mathcal{R}^*$ of the Radon transform operator $\mathcal{R}$ is the backprojection operator:
\[
	[\mathcal{R}^*g](u,v) = \int_0^\pi g(u\cos\theta+v\sin\theta,\theta)d\theta
	.
\]
If $g=\mathcal{R}f$, it can be shown that 
\[
	[\mathcal{R}^*\mathcal{R} f](u,v)
	= k(u,v) * f(u,v)
	= \int_{-\infty}^\infty \int_{-\infty}^\infty
		k(\zeta,\xi) f(u-\zeta,v-\xi) d\zeta d\xi
\]
for $k(u,v)=(u^2+v^2)^{-1/2}$ \cite{deans1983radon}.
Thus $\mathcal{R}^*\mathcal{R}$ is a linear spatially invariant operator that evaluates the convolution with the radial kernel $k$. 
Applying the Fourier transform operator $\mathcal{F}$ on both sides and then taking the inverse, we obtain
\begin{equation}\label{eq:radon-formula}
	\mathcal{R}^*\mathcal{R}f = \mathcal{F}^{-1}(\hat{k}\hat{f}),
\end{equation}
where $\hat{k}=\mathcal{F}k$ and $\hat{f}=\mathcal{F}f$ are the Fourier transforms of $k$ and $f$, respectively.

%\paragraph{Convex analysis.}

A function $g:\reals^n\rightarrow\reals$ is $L$-Lipschitz continuous if 
\[
|g(x)-g(y)|\le L\|x-y\|
\]
for all $x,y\in \reals^n$.
A function $g$ is convex if
\[
g(\theta x+(1-\theta)y)
\le
\theta g(x)+(1-\theta)g(y)
\]
for all $x,y\in \reals^n$ and $\theta\in[0,1]$.
The conjugate of $g$ is
\[
g^*(u)=\sup_{x\in \reals^n}u^T x-g(x).
\]
The proximal operator with respect to $g$ is
\[
\prox_{\alpha g}(z)=\argmin_{x\in \reals^n}\left\{
\alpha g(x)+\frac{1}{2}\|x-z\|^2
\right\}
\]
where $\alpha > 0$.
For many interesting functions, the proximal operator has a closed-form solution.
If $\prox_{(1/\alpha) g}$ is easy to compute, then so is $\prox_{\alpha g^*}$. This follows from Moreau's theorem,
\begin{equation}
\prox_{\alpha g^*}( x)=x-\alpha \prox_{(1/\alpha)g}((1/\alpha)x).
\label{eq:Moreau}
\end{equation}
The primal problem \eqref{eq:RCP} has the dual problem
\begin{equation}
\begin{array}{ll}
\underset{u\in \mathbb{R}^m}{\mbox{maximize}} &-g^*(u)-u^Tb\\
\mbox{subject to}&A^Tu=0.
\end{array}
\label{eq:dual}
\end{equation}
We say strong duality holds between \eqref{eq:RCP} and \eqref{eq:dual}
if their optimal values are the same.
%The algorithms \eqref{eq:PDHG}, \eqref{eq:DRS/ADMM}, and \eqref{eq:NCS}
%are primal-dual methods in that they find both primal and dual solutions.

Douglas--Rachford splitting (DRS) solves the optimization problem
of minimizing $f(x)+g(x)$ over $x\in \reals^n$
%\[
%\begin{array}{ll}
%\underset{x\in\reals^n}{\mbox{minimize}}&f(x)+g(x)
%\end{array}
%\]
with
\begin{align*}
x^{k+1/2}&=\prox_{\alpha g}(z^k)\\
x^{k+1}&=\prox_{\alpha f}(2x^{k+1/2}-z^k)\\
z^{k+1}&=z^k+x^{k+1}-x^k
\end{align*}
where $\alpha>0$.
DRS converges to a solution if a solution exists, a dual solution exists, and strong duality holds, 
where the dual problem is maximizing $-f^*(-\nu)-g^*(\nu)$ over $\nu\in\reals^n$.

\section{Convergence analysis}
\label{s:theory}
In this section, we analyze the convergence of NCS under the following  assumptions:
\begin{gather*}
\text{$g$ is convex and $\prox_{\alpha g^*}$ is well defined.}
\tag{A1}
\label{assump1}\\
\text{A primal-dual solution pair exists and strong duality holds.}
\tag{A2}
\label{assump2}\\
M\succeq \alpha A^TA
\tag{A3}
\label{assump3}
\end{gather*}
Assumptions \eqref{assump1} and  \eqref{assump2} are standard.
Convexity of $g$ is a strong assumption, but it holds in many applications of interest.
The other parts of \eqref{assump1} and  \eqref{assump2} are made to rule out pathologies.
Assumption \eqref{assump3} is strong as it requires the entire spectrum of $M$ to dominate the spectrum of $\alpha A^TA$, but such assumptions are commonly required for preconditioned nonlinear iterations \cite{cai_deblurring2016}.
%That $\prox_{\alpha g^*}$ is well-defined is a minor assumption made to rule out pathologies.
%%We could have instead assume $g$ is closed and proper.
%Assumption \eqref{assump2} is very standard, and is made to rule out pathologies.
%Almost all convergence results of PDHG, DRS, and ADMM assume \eqref{assump2},
%usually by assuming an appropriate Lagrangian has a saddle-point.
%Assumption \eqref{assump3} is really the main assumption.

We now state our main results.
\begin{theorem}
\label{thm:conv}
Assume \eqref{assump1}, \eqref{assump2}, and \eqref{assump3}.
Then NCS converges in that
$x^k\rightarrow x^\star$ and $u^k\rightarrow u^\star$,
 where $x^\star$ and $u^\star$ are primal and dual solutions.
\end{theorem}

% \begin{theorem}
% \label{thm:z-rate}
% Assume \eqref{assump1}, \eqref{assump2}, and \eqref{assump3}.
% Define
% \[
% z^{k+1}=A(2x^{k+1}-x^k)-b-\frac{1}{\alpha}(u^{k+1}-u^{k}).
% \]
% Then 
% \begin{align*}
% &g(z^{k+1})-g(z^\star)\\
% &\le \frac{1}{\alpha\sqrt{k+1}}
% \bigg(
% \|u^0-u^\star-\alpha A(x^0-x^\star)\|+\|u^\star\|+
% \|x^0-x^\star\|_{(\alpha M-\alpha^2 A^TA)}
% \bigg)^2
% \end{align*}
% \end{theorem}

\begin{theorem}
\label{thm:main-rate}
Assume \eqref{assump1}, \eqref{assump2}, and \eqref{assump3}.
Furthermore, assume $g$ is $L$-Lipschitz continuous. Then 
NCS converges with the rate 
%\begin{align*}
%&g(Ax^{k+1}-b)-g(Ax^\star-b)\\
%&\qquad\le \frac{1}{\alpha\sqrt{k+1}}
%D(x^{0}-x^\star,u^{0}-u^\star)\left(D(x^{0}-x^\star,u^{0}-u^\star)+\|u^\star\|+L\right)
%\end{align*}
\begin{align*}
&g(Ax^{k+1}-b)-g(Ax^\star-b)\\
&\le \frac{1}{\alpha\sqrt{k+1}}
\bigg(
\|u^0-u^\star-\alpha A(x^0-x^\star)\|+\|u^\star\|+L+
\|x^0-x^\star\|_{(\alpha M-\alpha^2 A^TA)}
\bigg)^2.
\end{align*}

\end{theorem}

\subsection{Discussion}
Theorem~\ref{thm:conv} establishes convergence
%, Theorem~\ref{thm:z-rate} establishes a rate of convergence with the $z^k$-iterate,
and
Theorem~\ref{thm:main-rate} establishes a rate of convergence
with the additional assumption of Lipschitz continuity,
which holds locally around the solution for the CT imaging problem.

For a given $\alpha$, the bound of Theorem~\ref{thm:main-rate} is minimized when $M=\alpha A^TA$, which corresponds to ADMM.
When $M\approx \alpha A^TA$, we can expect NCS to converge with a similar number of iterations compared to ADMM.
%In fact, since $M= \alpha A^TA$ corresponds to ADMM and $M=\gamma I$ corresponds to PDHG,
Theorem~\ref{thm:main-rate} to a certain extent explains why ADMM tends to converge in fewer iterations than PDHG.

At the same time, we would like to point out that the rate of Theorem~\ref{thm:main-rate} should not be considered a tight estimate.
The theoretical convergence rates of splitting methods like PDHG and ADMM 
are upper bounds, and they rarely track the empirical rate;
%In general, sublinear rates of the form $\mathcal{O}(1/k)$ or $\mathcal{O}(1/\sqrt{k})$ in optimization often
the rates come from a worst-case analysis, and the worst case rarely occurs.
% In general, sublinear rates of the form
% $\mathcal{O}(1/k)$ or $\mathcal{O}(1/\sqrt{k})$ 
% in optimization 
% is a worst-case analysis, and this worst-case rarely occurs.
Nevertheless, the qualitative message of Theorem~\ref{thm:main-rate} persists in experiments.

The rate of Theorem~\ref{thm:main-rate} is of order $\mathcal{O}(1/\sqrt{k})$,
rather than $\mathcal{O}(1/k)$,
because we analyze convergence of the actual iterates, rather than
the ``ergodic'' iterates.
It is possible to perform a similar analysis with the ergodic iterates
to obtain a $\mathcal{O}(1/k)$-rate and have a constant
depending on $\alpha M-\alpha^2 A^TA$ in a similar manner.

%When a convergence analysis does not actually track the actual performance of an algorithm,
%the convergence analysis does not really provide a solid basis for deciding which algorithm is good.

%Theorem~\ref{thm:main-rate} states that we want $M\approx \alpha A^TA$. This is not the same 
%as $M^+\approx (A^TA)^+$ when $A^TA$ is not invertible.

\subsection{Proofs}
%\subsection{Proof of Theorems~\ref{thm:conv} and \ref{thm:main-rate}}
We now prove Theorems~\ref{thm:conv} and \ref{thm:main-rate}.
Assume \eqref{assump1}, \eqref{assump2}, and \eqref{assump3} throughout this section.
In \cite{connor2017}, O'Connor and Vandenberghe presented a reduction of
PDHG to DRS. Our proof relies on a similar reduction inspired by theirs.

We write $x^\star$ for a primal solution and $u^\star$ for a dual solution throughout.
Define the seminorm $D(x,u):\reals^{n+m}\rightarrow\reals$ with
\begin{align*}
D^2(x,u)
&=
\begin{bmatrix}
x^T&u^T
\end{bmatrix}
\begin{bmatrix}
\alpha M&-\alpha A^T\\
-\alpha A& I
\end{bmatrix}
\begin{bmatrix}
x\\u
\end{bmatrix}\\
&=\|u-\alpha Ax\|^2+
\|x\|_{(\alpha M-\alpha^2 A^TA)}^2.
\end{align*}
To clarify, we mean $D^2(x,u)=(D(x,u))^2$.
Let $B\in \reals^{n\times n}$ be the positive semidefinite matrix satisfying
\[
B^2=(1/\alpha) M-A^TA,
\]
i.e., $B$ is the matrix square root of $(1/\alpha) M-A^TA$.
For any set $S$, define the indicator function
\[
\delta_{S}(x)=\left\{
\begin{array}{ll}
0&\text{if }x\in S\\
\infty&\text{otherwise}.
\end{array}
\right.
\]
We say a mapping $T:\mathbb{R}^\ell\rightarrow\mathbb{R}^\ell$ is firmly nonexpansive if
\[
\|T(a)-T(b)\|^2\le
\|a-b\|^2-\|(a-T(a))-(b-T(b))\|^2
\]
for any $a,b\in \reals^\ell$.

%Consider the following problem, which is equivalent to the dual problem 
%\[
%\underset{u \in \reals^m}{\mbox{minimize}}\qquad
%\delta_{\{0\}}(-A^Tu)+g^*(u)+b^Tu
%\]
%which is in turn equivalent to

%We now reduce our method to Douglas-Rachford splitting (DRS).
Consider the primal-dual problem pair
\begin{gather}
\underset{z \in \reals^m,\tilde{z}\in \reals^n}{\mbox{minimize}}\qquad
g(z)+\delta_{\{(Ax-b,Bx)\,|\,x\in \reals^n\}}(z,\tilde{z})
\label{eq:p-prime}\\
\underset{u ,\tilde{u}\in \reals^m}{\mbox{maximize}}\qquad
-g^*(u)-\delta_{\{0\}}(-A^Tu-B^T\tilde{u})-b^Tu-\delta_{\{0\}}(\tilde{u}),
\label{eq:d-prime}
\end{gather}
%Clearly \eqref{eq:p-prime} and \eqref{eq:d-prime}
which are equivalent to the primal-dual problem pair \eqref{eq:RCP} and \eqref{eq:dual}.
We apply DRS to the equivalent dual problem \eqref{eq:d-prime} to get
\begin{align}
x^{k+1}&=-\frac{1}{\alpha}(A^TA+B^TB)^{+}(A^T(v^k-\alpha b)+B^T\tilde{v}^k)
\nonumber\\
u ^{k+1/2}&=v^k+\alpha (Ax^{k+1}-b)\nonumber\\
\tilde{u}^{k+1/2}&=\tilde{v}^k+\alpha B x^{k+1}\nonumber\\
u^{k+1}&=\prox_{\alpha g^*}(v^k+2\alpha Ax^{k+1})
\label{eq:main-iter}
\\
\tilde{u}^{k+1}&=0\nonumber\\
v^{k+1}&=v^k+u^{k+1}-u^{k+1/2}\nonumber\\
\tilde{v}^{k+1}&=\tilde{v}^k+\tilde{u}^{k+1}-\tilde{u}^{k+1/2}.\nonumber
\end{align}
We can simplify this to
%\begin{align*}
%x^{k+1}&=-\frac{1}{\alpha}(A^TA+B^TB)^{+}(A^Tv^k+B^T\tilde{v}^k)\\
%u^{k+1}&=\prox_{\alpha g^*}(v^k+2\alpha Ax^{k+1}-\alpha b)\\
%v^{k+1}&=u^{k+1}-\alpha Ax^{k+1}\\
%\tilde{v}^{k+1}&=-\alpha Bx^{k+1}
%\end{align*}
%Finally, we get
\begin{align}
x^{k+1}&=x^k-\frac{1}{\alpha}(A^TA+B^TB)^{+}A^Tu^k\nonumber\\
u^{k+1}&=\prox_{\alpha g^*}(u^k+\alpha A(2x^{k+1}-x^k)-\alpha b)\nonumber\\
v^{k+1}&=u^{k+1}-\alpha( Ax^{k+1}-b)\label{eq:main-iter2}\\
\tilde{v}^{k+1}&=-\alpha Bx^{k+1}\nonumber.
\end{align}
Note that \eqref{eq:main-iter2} is NCS with the additional $v^k$- and $\tilde{v}^k$-iterates, which do not affect the $x^k$- and $u^k$-iterates.
\begin{fact}
\label{lem:fixed-points}
Under assumption \eqref{assump2},
iteration \eqref{eq:main-iter} has fixed points
\[
(v^\star,\tilde{v}^{\star})=
(u^\star-\alpha (Ax^\star-b),-\alpha Bx^\star)
\]
where $x^\star$ and $u^\star$ are any primal and dual solutions
to \eqref{eq:RCP} and \eqref{eq:dual}.\end{fact}
\begin{proof}
Although Fact~\ref{lem:fixed-points} follows from standard arguments of DRS theory,
we are not aware of a single theorem that we can directly cite for Fact~\ref{lem:fixed-points}.
We therefore briefly lay out this standard argument.

By assumption \eqref{assump2}, 
the primal-dual problem pair \eqref{eq:RCP} and \eqref{eq:dual}
has a primal-dual solution pair $x^\star$ and $u^\star$ 
and strong duality holds.
The equivalent primal-dual problem pair \eqref{eq:p-prime} and \eqref{eq:d-prime}
therefore 
has a primal-dual solution pair $(Ax^\star-b,Bx^\star)$ and $(u^\star,0)$
and strong duality holds.
Combining Theorem~7.1 of \cite{bauschke2012}, Fact~3.3 of \cite{Bauschke2017},
and standard convex duality,
while remembering that \eqref{eq:main-iter} is 
DRS applied to the equivalent dual problem \eqref{eq:d-prime},
 we arrive at the stated result.
%$(x^\star,u^\star)$ is a primal-dual solution if and only if it is a saddle point of
%\[
%\mathcal{L}=\langle u,Ax-b\rangle -g^*(u).
%\]
\end{proof}

\begin{proof}[Proof of Theorem \ref{thm:conv}]
%Consider any fixed point $(v^\star,\tilde{v}^{\star})=
%(u^\star-\alpha Ax^\star,-\alpha Bx^\star)$
%of \eqref{eq:main-iter}.
%Then by Fact~\ref{fact:fne},
%\[
%\|v^\star-v^{k+1}\|^2+
%\|\tilde{v}^\star-\tilde{v}^{k+1}\|^2
%\le 
%\]
The DRS iteration represents a fixed-point iteration with respect to a 
firmly nonexpansive mapping; i.e.,
the mapping $(v^{k},\tilde{v}^k)\mapsto (v^{k+1},\tilde{v}^{k+1})$
is firmly nonexpansive.
A fixed-point iteration with respect to a firmly nonexpansive mapping
converges to a fixed point provided a fixed point exists
\cite{mann1953,krasnoselskii1955,BauschkeCombettes2017_convex}.
So with Fact~\ref{lem:fixed-points}, we conclude
\[
v^k\rightarrow v^\star=u^\star-\alpha (Ax^\star-b)\qquad 
\tilde{v}^k\rightarrow \tilde{v}^{\star}= -\alpha Bx^\star
\]
for some primal and dual solutions $x^\star$ and $u^\star$.

Substituting this into \eqref{eq:main-iter}, we get
\[
x^{k}\rightarrow \underbrace{(A^TA+B^TB)^+(A^TA+B^TB)x^\star}_{:=\bar{x}^\star}.
\]
Note that $A\bar{x}^\star=Ax^\star$, since
\[
(A^TA+B^TB)^+(A^TA+B^TB)|_{\mathcal{N}(A)^\perp}=I|_{\mathcal{N}(A)^\perp}.
\]
So if $x^\star$ is optimal for the primal problem, then so is $\bar{x}^\star$.

Plugging these into \eqref{eq:main-iter}, we get
\[
u^k\rightarrow \prox_{\alpha g^*}(u^\star+\alpha (A\bar{x}^\star-b)).
\]
Standard arguments from convex duality or Fact~\ref{lem:fixed-points}
tell us that
\[
\prox_{\alpha g^*}(u^\star+\alpha (A\bar{x}^\star-b))=u^\star,
\]
and we conclude $u^k\rightarrow u^\star$.
\end{proof}

In the proofs of Lemmas~\ref{lem:monotone} and \ref{lem:summable} we use the identities
\begin{align*}
D^2(x^{k+1}-x^k,u^{k+1}-u^k)
&=\|v^{k+1}-v^k\|^2+\|\tilde{v}^{k+1}-\tilde{v}^k\|^2\\
D^2(x^{k}-x^\star,u^{k}-u^\star)
&=\|v^{k}-v^\star\|^2+\|\tilde{v}^{k}-\tilde{v}^\star\|^2.
\end{align*}

\begin{lemma}
\label{lem:monotone}
The two sequences
\[
D(x^{k+1}-x^k,u^{k+1}-u^k)
\quad\text{and}\quad
D(x^{k}-x^\star,u^{k}-u^\star)
\]
are nonincreasing with $k=0,1,\dots$.
%\|u^{k+1}-\alpha Ax^{k+1}-u^{k}-\alpha Ax^{k}\|^2+\|x^{k+1}-x^k\|^2_{M-A^TA}
\end{lemma}
\begin{proof}
DRS maps 
$(v^{k},\tilde{v}^k)\mapsto (v^{k+1},\tilde{v}^{k+1})$ 
and 
$(v^{k+1},\tilde{v}^{k+1})\mapsto (v^{k+2},\tilde{v}^{k+2})$.
Since DRS is a nonexpansive iteration, we have
\[
\|v^{k+2}-v^{k+1}\|^2+
\|\tilde{v}^{k+2}-\tilde{v}^{k+1}\|^2
\le 
\|v^{k+1}-v^{k}\|^2+
\|\tilde{v}^{k+1}-\tilde{v}^{k}\|^2,
\]
which is the first stated result.
Likewise, DRS maps 
$(v^{k},\tilde{v}^k)\mapsto (v^{k+1},\tilde{v}^{k+1})$ 
and 
$(u^\star-\alpha (Ax^\star-b),-\alpha Bx^\star)$ is a fixed point.
With the same reasoning, we get the second stated result.
%Since DRS is a nonexpansive iteration, we have
%\[
%\|v^{k+2}-((u^\star-\alpha (Ax^\star-b))\|^2+
%\|\tilde{v}^{k+2}+\alpha Bx\|^2
%\le 
%\|v^{k+1}-((u^\star-\alpha (Ax^\star-b))\|^2+
%\|\tilde{v}^{k+1}+\alpha Bx\|^2,
%\]
%which is the first stated result.
\end{proof}

\begin{lemma}
\label{lem:summable}
\begin{align*}
D^2(x^{k+1}-x^k,u^{k+1}-u^k)\le
\frac{1}{k+1}D^2(x^{0}-x^\star,u^{0}-u^\star)
\end{align*}
\end{lemma}
\begin{proof}
Since DRS is a firmly nonexpansive iteration, we have
\[
\|v^{k+1}-v^\star\|^2+\|\tilde{v}^{k+1}-\tilde{v}^\star\|^2\le 
\|v^{k}-v^\star\|^2+\|\tilde{v}^{k}-\tilde{v}^\star\|^2-
\left(\|v^{k+1}-v^k\|^2+\|\tilde{v}^{k+1}-\tilde{v}^k\|^2\right)
\]
which we can rewrite as
\[
D^2(x^{k+1}-x^\star,u^{k+1}-u^\star)
\le 
D^2(x^{k}-x^\star,u^{k}-u^\star)-
D^2(x^{k+1}-x^k,u^{k+1}-u^k).
\]
Summing this, we get
\begin{align*}
\sum^k_{i=0}
D^2(x^{i+1}-x^i,u^{i+1}-u^i)&\le 
D^2(x^{0}-x^\star,u^{0}-u^\star)-
D^2(x^{k+1}-x^\star,u^{k+1}-u^\star)\\
&\le 
D^2(x^{0}-x^\star,u^{0}-u^\star).
\end{align*}
Since $D^2(x^{k+1}-x^k,u^{k+1}-u^k)$ is nonincreasing,
by Lemma~\ref{lem:monotone}, we have
\begin{align*}
(k+1)
D^2(x^{k+1}-x^k,u^{k+1}-u^k)\le
\sum^k_{i=0}
D^2(x^{i+1}-x^i,u^{i+1}-u^i).
\end{align*}
The stated result follows from combining these inequalities.
\end{proof}

\begin{lemma}
\label{lem:z-rate}
Define
\[
z^{k+1}=A(2x^{k+1}-x^k)-b-\frac{1}{\alpha}(u^{k+1}-u^{k}).
\]
Then
\begin{align*}
g(z^{k+1})-g(Ax^\star-b)\le \frac{1}{\alpha\sqrt{k+1}}
D(x^{0}-x^\star,u^{0}-u^\star)\left(D(x^{0}-x^\star,u^{0}-u^\star)+\|u^\star\|\right).
\end{align*}
\end{lemma}
\begin{proof}
Since $u^{k+1}$ is defined as a minimizer, we have
\[
u^k+\alpha A(2x^{k+1}-x^k)-\alpha b
\in
u^{k+1}+\alpha \partial g^*(u^{k+1}).
\]
This implies 
$z^{k+1}\in \partial g^*(u^{k+1})$
and 
$u^{k+1}\in \partial g(z^{k+1})$.
This gives us the subgradient inequality
\begin{equation}
g(z^{k+1})-g(z)\le \langle u^{k+1},z^{k+1}-z\rangle
\qquad\forall z\in \reals^m.
\label{eq:subgrad}
\end{equation}

First, note that
\begin{align}
&\alpha \langle B(x^{k+1}-x^k),B(x^{k+1}-x^\star)\rangle
+\alpha \langle A(x^{k+1}-x^k),A(x^{k+1}-x^\star)\rangle\nonumber\\
&\qquad\qquad+\langle u^k,A(x^{k+1}-x^\star)\rangle\nonumber\\
&\quad=
-\langle x^{k+1}-x^\star,
((A^TA+B^TB)(A^TA+B^TB)^+-I)A^Tu^k\rangle\nonumber\\
&\quad=0.\label{eq:lemma-eq4}
\end{align}
The first equality follows from
the $x^{k+1}$-update of \eqref{eq:main-iter2}.
The second equality follows from the fact that
$((A^TA+B^TB)(A^TA+B^TB)^+-I)v=0$ for any $v\in \mathrm{range}(A^T)+\mathrm{range}(B^T)$
and $A^Tu^k\in \mathrm{range}(A^T)\subseteq \mathrm{range}(A^T)+\mathrm{range}(B^T)$.

Using this, we get
\begin{align*}
g&(z^{k+1})-g(Ax^\star-b)\le \langle u^{k+1},z^{k+1}-Ax^\star+b\rangle\\
%&=
%\langle u^{k+1},
%A(2x^{k+1}-x^k-x^\star)+(1/\alpha)(u^k-u^{k+1})\rangle\\
%&=
%-\frac{1}{\alpha}
%\langle u^{k+1},
%u^{k+1}-u^k
%-\alpha A(x^{k+1}-x^\star)\rangle
%+
%\langle u^{k+1},
%A(x^{k+1}-x^k)\rangle\\
&=
-\frac{1}{\alpha}
\langle
u^{k+1}-u^k-\alpha A(x^{k+1}-x^k),
u^{k+1}-\alpha A(x^{k+1}-x^\star)
\rangle\\
&\qquad+\langle u^k,A(x^{k+1}-x^\star)\rangle+\alpha \langle A(x^{k+1}-x^\star),A(x^{k+1}-x^k)\rangle\\
&=
-\frac{1}{\alpha}
\langle
u^{k+1}-u^k-\alpha A(x^{k+1}-x^k),
u^{k+1}-\alpha A(x^{k+1}-x^\star)
\rangle\\
&\qquad-\frac{1}{\alpha}\langle \alpha B(x^{k+1}-x^k),\alpha B(x^{k+1}-x^\star)\rangle\\
&=-
\frac{1}{\alpha}
\begin{bmatrix}
x^{k+1}-x^k\\u^{k+1}-u^k
\end{bmatrix}^T
\begin{bmatrix}
\alpha M&-\alpha A^T\\
-\alpha A& I
\end{bmatrix}
\begin{bmatrix}
x^{k+1}-x^\star\\u^{k+1}
\end{bmatrix}\\
&\le \frac{1}{\alpha}
D(x^{k+1}-x^k,u^{k+1}-u^k)D(x^{k+1}-x^\star,u^{k+1})\\
&\le \frac{1}{\alpha}
D(x^{k+1}-x^k,u^{k+1}-u^k)\left(D(x^{k+1}-x^\star,u^{k+1}-u^\star)+\|u^\star\|\right)\\
&\le \frac{1}{\alpha\sqrt{k+1}}
D(x^{0}-x^\star,u^{0}-u^\star)\left(D(x^{0}-x^\star,u^{0}-u^\star)+\|u^\star\|\right).
\end{align*}
The first inequality is \eqref{eq:subgrad}.
The first equality follows from substituting the definition of $z^{k+1}$ and reorganizing the terms.
The second equality follows from \eqref{eq:lemma-eq4}.
The third equality follows from the definition of $B$.
The second inequality follows from %\eqref{eq:CS}, 
the Cauchy--Schwartz inequality.
The third inequality follows from the triangle inequality
and the fact that $D(0,u^\star)=\|u^\star\|$.
The finally inequality follows from Lemmas~\ref{lem:monotone} and \ref{lem:summable}.
\end{proof}

\begin{lemma}
\label{lem:main-rate}
Further assume $g$ is $L$-Lipschitz.
Then 
\begin{align*}
&g(Ax^{k+1}-b)-g(Ax^\star-b)\\
&\le \frac{1}{\alpha\sqrt{k+1}}
D(x^{0}-x^\star,u^{0}-u^\star)\left(D(x^{0}-x^\star,u^{0}-u^\star)+\|u^\star\|+L\right).
\end{align*}
\end{lemma}
\begin{proof}
Using the assumption that $g$ is $L$-Lipschitz and Lemma~\ref{lem:monotone}, we have
\begin{align*}
g(Ax^{k+1}-b)-g(z^{k+1})
&\le
|g(Ax^{k+1}-b)-g(z^{k+1})|\\
&\le L\|Ax^{k+1}-b-z^{k+1}\|\\
&=L\|A(x^{k+1}-x^k)-(1/\alpha)(u^{k+1}-u^k)\|\\
&\le\frac{L}{\alpha}D(x^{k+1}-x^k,u^{k+1}-u^k)\\
&\le
\frac{L}{\alpha\sqrt{k+1}}D(x^{0}-x^\star,u^{0}-u^\star)
\end{align*}
Combining this with Lemma~\ref{lem:z-rate} gives us the stated result.
\end{proof}

\begin{proof}[Proof of Theorem~\ref{thm:main-rate}]
The rate of Theorem~\ref{thm:main-rate} is a simpler version of
the rate of Lemma~\ref{lem:main-rate}.
When $a,b\ge 0$ are scalars, we have $\sqrt{a^2+b^2}\le a+b$. With this, we have
\begin{align*}
&D(x^{0}-x^\star,u^{0}-u^\star)\left(D(x^{0}-x^\star,u^{0}-u^\star)+\|u^\star\|+L\right)\\
&\quad\le \left(D(x^{0}-x^\star,u^{0}-u^\star)+\|u^\star\|+L\right)^2\\
&\quad\le 
\bigg(
\|u^0-u^\star-\alpha A(x^0-x^\star)\|+\|u^\star\|+L+
\|x^0-x^\star\|_{(\alpha M-\alpha^2 A^TA)}
\bigg)^2.
\end{align*}
Combining this with Lemma~\ref{lem:main-rate} gives us the stated result.
\end{proof}

\section{Applications}
\label{s:appl}
In this section, we present applications that benefit from NCS.
%parallel beam geometry CT and statistical reconstruction PET.
%In both applications, the measurement matrix is or related to the discrete Radon transform.
%However, NCS is more generally applicable, since near-circulant systems do arise in other setups
%such as fan beam geometry CT.

% Consider an imaging system with a measurement matrix $L$ and its continuous counterpart $\mathcal{L}$.
% %that can be understood as the discretization of a continuous linear operator $\mathcal{L}$.
% If $\mathcal{L}^*\mathcal{L}$ is spatially invariant, and we know its point spread function,
% then we can apply NCS in the same manner as we do in this section.

\subsection{Computed tomography}
In X-ray CT, X-ray beams are illuminated along multiple lines from a rotating source to a rotating set of detectors.
When the beams pass through the object of interest, energy is absorbed, and the intensity of the beams reduces.
This experimental setup corresponds to measuring $b=Ex\in \reals^m$, 
where $E\in \reals^{m\times n}$ is the measurement matrix induced by the geometry of the detector array, and $x\in \reals^n$ is the vectorized image of interest.
%$E\in \reals^{m\times n}$ is the discretization of the Radon transform \eqref{eq:radon-def} induced by the geometry of the detector array, and $x\in \reals^n$ is the vectorized image of interest.
%One way to recover the image is to solve the least squares problem
%\[
%\begin{array}{ll}
%\underset{x\in \reals^n}{\mbox{minimize}}&
%\frac{1}{2}\|Ex-b\|^2.
%\end{array}
%\]

When the signal-to-noise ratio is poor or when there are fewer measurements than the number of unknowns,
it is helpful to solve the least-squares problem with total variation regularization \cite{rudin1992nonlinear,choi2010compressed,bian2014investigation}
%Using total variation penalty \cite{rudin1992nonlinear}, a popular choice, leads to the nonsmooth optimization problem
\[
\begin{array}{ll}
\underset{x\in \reals^n}{\mbox{minimize}}&
\frac{1}{2}\|Ex-b\|^2+\lambda \|Dx\|_1
\end{array}
\]
where the optimization variable $x\in \reals^n$ represents the image (2D or 3D) to recover,
$D$ is the (2D or 3D) finite difference matrix, and $\lambda>0$.
This is equivalent to 
\[
\begin{array}{ll}
\mbox{minimize}&
\frac{1}{2}\|y\|^2+(\lambda\alpha/\beta) \|z\|_1\\
\mbox{subject to}&
\begin{bmatrix}
E\\(\beta/\alpha)D
\end{bmatrix}
x-
\begin{bmatrix}
b\\0
\end{bmatrix}=
\begin{bmatrix}
y\\z
\end{bmatrix}
\end{array}
\]
for any $\alpha,\beta>0$.

NCS applied to this problem is
\begin{align*}
x^{k+1}&=x^k
-(1/\alpha)F^{-1}(\diag(h) F(\alpha E^Tu^k+\beta D^Tv^k))
\\
%(A^Tu^k+D^Tv^k)\\
u^{k+1}&=
\frac{1}{1+\alpha}(u^k+\alpha E(2x^{k+1}-x^k)-\alpha b)\\
v^{k+1}&=\Pi_{[-\lambda\alpha/\beta,\lambda\alpha/\beta]}\left(v^k+\beta D(2x^{k+1}-x^k)\right),
\end{align*}
where $\alpha,\beta>0$.
Here, $F$ represents the DFT (to be evaluated with the FFT),
$\diag(h)=FM^+F^{-1}$ where $M$ is the circulant approximation used for NCS,
and $\Pi_{[-s,s]}$ is evaluated componentwise with
\[
\Pi_{[-s,s]}(x)=
\left\{
\begin{array}{ll}
-s&\text{for }x\le -s\\
x&\text{for } -s<x<s\\
s
&\text{for } s\le x.
\end{array}
\right.
\]

%\paragraph{Construction of $M$.}

\subsubsection{Parallel beam geometry}
\label{ss:parbeam}
In CT with 2D parallel beam geometry, parallel X-rays are illuminated through the object and detectors measure their attenuation.
The X-ray source and detectors circle around the object.
Under this setup, $E=R\in \reals^{m\times n}$ is the discretization of the Radon transform \eqref{eq:radon-def}.
%that $M$ should satisfy $M\approx \alpha A^TA$ and $M\succeq \alpha A^TA$.
For simplicity, we consider the 2D $N\times N$ grid, with $N^2=n$. 
%We identify $h\in \reals^{N^2}$ with the $N\times N$ mask $H$ defined with
%\begin{align}
%H_{(j+1)(k+1)}&=
%\gamma+
%C \alpha
%\left(
%\min\{j,N-j\}^2+\min\{k,N-k\}^2\right)^{-1/2}\nonumber\\
%&\qquad\qquad\qquad\qquad+
%4\beta^2/\alpha\left(\sin^2\left(
%\frac{j\pi}{ N}\right)
%+
% \sin^2\left(\frac{k\pi}{N}\right)\right)
%\label{eq:choice-M}
%\end{align}
We choose $M=F^{-1}\diag(h)^+F$
based on \eqref{eq:f-diff} and \eqref{eq:radon-formula}.
We approximate $R^TR\approx F^{-1}\diag(h_R)F$, where 
we identify $h_R\in \reals^{N^2}$ with the $N\times N$ mask $H_R$ defined with
\begin{equation}
(H_R)_{(j+1),(k+1)}=
C_R
\left(
\min\{j,N-j\}^2+\min\{k,N-k\}^2\right)^{-1/2}
 \label{eq:circ_R}
\end{equation}
for $j,k=0,\dots,N-1$, except for $j=k=0$.
The constant $C_R$ accounts for the scaling factor
of the discrete Radon transform and backprojection.
We separately tune the ``DC component'' (direct current) $H_{1,1}$.
We approximate $D^TD\approx F^{-1}\diag(h_D)F$, where 
we identify $h_D\in \reals^{N^2}$ with the $N\times N$ mask $H_D$ defined with
\begin{equation}
(H_D)_{(j+1),(k+1)}=
4\left(\sin^2\left(
\frac{j\pi}{ N}\right)
+
 \sin^2\left(\frac{k\pi}{N}\right)\right)
 \label{eq:circ_D}
\end{equation}
for $j,k=0,\dots,N-1$.
%We separately tune the ``DC component'' $H_{1,1}$.
%The constant $C$ accounts for the scaling factor
%of the discrete Radon transform and backprojection.

\emph{Remark.}
In the continuous space, before discretization, we have
\[
(\mathcal{R}^*\mathcal{R})^{-1}f=
\mathcal{F}^{-1}\left(\sqrt{\zeta^2+\xi^2}\hat{f}(\zeta,\xi)\right).
\]
Since the $\sqrt{\zeta^2+\xi^2}$ factor vanishes at $(0,0)$, the DC component of the original image is lost.
In the discretized space, however, the DC component is not fully lost, and approximating the DC component of $R^TR$ with $0$, which corresponds to using \eqref{eq:circ_R} for $H_{1,1}$, leads to poor results.
In our experiments, we observed that tuning the value of $H_{1,1}$ separately resolved this issue.

\emph{Remark.}
The pseudoinverse of $D^TD$,
the discrete Laplacian operator with Neumann boundary conditions, can be efficiently computed
with the discrete cosine transform (DCT),
provided one uses a fast DCT algorithm
\cite{ahmed1974,strang1999,golub2012,neumann_bc_1999,zhao_hyperspectral2013}.
However, this fact does not help us, since 
we need the pseudoinverse of $R^TR+D^TD$,
and the DCT basis does not simultaneously diagonalize $R^TR$ and $D^TD$.

\subsubsection{Fan beam and cone beam geometry}
In CT with 2D fan beam geometry, X-rays emanate from a point and detectors on the other side receive beams within a certain angle (so the detected beams form a fan shape). The X-ray source and detectors circle around the object.
The measurement operator is a discretization of the divergent beam transform, which is related to the Radon transform through reparameterization \cite{netterer2001}.

In CT with 3D cone beam geometry, X-rays emanate from a point and detectors on the other side receive beams within a certain solid angle (so the detected beams form a cone shape).
The X-ray source and detectors traverse an orbit in the 3D space. The extent to which the linear measurement operator leads to a near-circulant linear system depends on this orbit \cite{katsevich2002}.

If $C$ is circulant with $C=F^{-1}\diag(h)F$, then $FCv=\diag(h) Fv$ and
\[
h_i=(FCv)_i/(Fv)_i,
\]
for $i=1,\dots,n$ for any $v\in \reals^n$.
Therefore, we approximate $E^TE\approx F^{-1}\diag(h_E)F$ by taking the average of 
\begin{equation}
(FE^TEv)_i/(Fv)_i
\label{eq:H-fan}
\end{equation}
where $i$ represents all 2D or 3D indices, for many $v\in \reals^n$.
For $D^TD$, we use the circulant approximation of \eqref{eq:circ_D} for the 2D fan beam geometry, and an analogous circulant approximation for the 3D cone beam geometry.

%When we apply $S$ and $\Pi_{[-\lambda,\lambda]}$ to vectors/images, we apply the function term-by-term.

% To clarify, the $F^{-1}(H\odot Fy^{k+1})$ involves taking the 2D  FFT of the image $y^{k+1}$,
% performing an element-wise multiplication with $\diag(h)$, taking the 2D inverse FFT,
% and discarding the imaginary part.

%\begin{align*}
%x^{k+1}&=x^k
%-(\gamma I+\alpha A_1^TA_1+\alpha D_1^TD_1)^{-1}(A^Tu^k+D^Tv^k)\\
%u^{k+1}&=
%\frac{1}{1+\alpha}(u^k+\alpha A(2x^{k+1}-x^k)-\alpha b)\\
%v^{k+1}&=\Pi_{[-\lambda,\lambda]}\left(v^k+\alpha D(2x^{k+1}-x^k)\right)
%\end{align*}

% \emph{Remark.}
% The idea of 
% In these algorithms, we have 2 linear systems, and we scale 
% We scale $D\mapsto (\beta/\alpha) D$, 
% which is a technique explored in  \cite{pock2011}.

\subsection{Positron emission tomography}
Positron emission tomography (PET) images positron-emitting isotopes placed within an object \cite{jain1989fundamentals}. When a positron is emitted, it combines with a nearby electron to emit two gamma-ray photons in opposite directions with a uniformly random orientation. An array of detectors surround the object and detect the pair of gamma rays in coincidence.
The emissions follow a Poisson process with intensity proportional to the concentration of the positron-emitting isotopes, which, in turn, is proportional to the tissue's metabolic activity. The intensity is zero outside the object.
The problem of reconstructing images corrupted by Poisson noise also arises in other setups as well \cite{huang_recent_2019}.

The region is discretized into a grid of $n$ pixels.
Let $x_j$ denote the intensity within pixel $j$ for $j=1,\dots,n$.
The reconstruction problem estimates $x\in \mathbb{R}^n$
from the observed data $b\in \mathbb{N}^m$, coincidence pair counts,
where $m$ is the number of detector pairs with which the lines of flight can be detected.
Let $e_{ij}$ be the probability that a photon pair emitted from pixel $j$ is
detected by the $i$th detector pair. 
This quantity is determined by the geometry of the experimental setup.
Each $b_i$ is a measurement of independent Poisson random variables
with mean $\sum_{j=1}^n e_{ij}x_j$ for $i=1,\dots,m$.
The maximum likelihood estimation of $x$ is
\begin{equation}\label{eq:pet_mle}
\begin{array}{ll}
\underset{x\in \reals^n}{\mbox{minimize}}&
\sum^m_{j=1}\ell((Ex)_j;b_j)
\end{array}
\end{equation}
where $E=(e_{ij})\in\mathbb{R}^{m\times n}$,
\[
\ell(y;b)=
\left\{
\begin{array}{ll}
y -b\log y& b>0\\
y+\delta_{\reals_+}(y)&b=0,
\end{array}\right.
\]
and 
\[
\delta_{\mathbb{R}_+}(y)=
\left\{
\begin{array}{ll}
0&\text{for }y \ge 0\\
\infty&\text{otherwise.}
\end{array}\right.
\]
For reference, see \cite{shepp1982maximum,lange1984reconstruction,vardi1985statistical}.
Again, we use total variation regularization to improve the reconstruction.
This leads to the optimization problem
\[
\begin{array}{ll}
\underset{x\in \reals^n}{\mbox{minimize}}&
\sum^m_{j=1}\ell((Ex)_j;b_j)+\lambda \|Dx\|_1,
\end{array}
\]
where $\lambda>0$.

Similar to before, NCS applied to this problem is
\begin{align*}
x^{k+1}&= x^k-
(1/\alpha)F^{-1}(\diag(h) F(\alpha E^Tu^k+\beta D^Tv^k))\\
u^{k+1}_i&=S(u^k_i+\alpha (E(2x^{k+1}-x^k))_i;\alpha b_i)\quad \text{for } i=1,\dots,n\\
v^{k+1}&=\Pi_{[-\lambda\alpha/\beta,\lambda\alpha/\beta]}\left(v^k+\beta D(2x^{k+1}-x^k)\right),
\end{align*}
where $\alpha,\beta>0$ and $S$ is evaluated componentwise with
\[
S(u;c)
=1+\frac{u-1-\sqrt{(u-1)^2+4c}}{2}.
\]
%It turns out that $E$ is very similar to the Radon transform.
By considering moment matching instead of maximum likelihood in the estimation of $x$, it can be shown that $E$ is very close to the Radon transform \cite{shepp1982maximum,vardi1985statistical}.
%(This is the justification of Fourier reconstruction techniques such as convolution backprojection for PET.)
Therefore, %{\color{red}$E^TE\approx F^{-1}\diag(h)F$},
$E^TE\approx R^TR$, and we use the circulant approximation described in Section~\ref{ss:parbeam}.
For $D^TD$, we use the circulant approximation of \eqref{eq:circ_D}.

\emph{Remark.}
Although $\ell$ is a differentiable convex function, its domain is not closed and its gradient is not Lipschitz continuous. This makes gradient and projected gradient methods difficult to apply.
NCS does not run into this difficulty since it relies on the proximal operator of $\ell$.

\emph{Remark.}
Sometimes, one might want to incorporate a positivity constraint $x\ge 0$.
This can be done by using the linear system
\[
\begin{bmatrix}
E\\D\\I
\end{bmatrix}x-
\begin{bmatrix}
b\\0\\0
\end{bmatrix}=
\begin{bmatrix}
y\\z\\s
\end{bmatrix}
\]
and the objective function
\[
\sum^m_{j=1}\ell(y_j;b_j)+\lambda \|z\|_1+\delta_{\mathbb{R}_+}(s).
\]

%We can also choose to ignore the positivity constraint and solve

%%\[
%%\begin{array}{ll}
%%\underset{x\in \reals^n}{\mbox{minimize}}&
%%\sum^n_{i=1}g((Ax)_i;b_i)+\lambda \|Dx\|_1\\
%%&x_i\ge 0\quad\text{ for }i=1,\dots,n
%%\end{array}
%%\]
%%
%%This is equivalent to
%%\[
%%\begin{array}{ll}
%%\underset{x\in \reals^n}{\mbox{minimize}}&
%%\sum^n_{i=1}
%%g(y_i;b_i)+\lambda |z_i|+\delta_{\reals_+}(x_i)\\
%%&
%%\begin{bmatrix}
%%A\\D\\I
%%\end{bmatrix}
%%x=
%%\begin{bmatrix}
%%y\\z\\w
%%\end{bmatrix}
%%\end{array}
%%\]
%%The main algorithm becomes
%%
%%\begin{align*}
%%x^{k+1}&= x^k-\left(
%%\beta I+\alpha \widetilde{A^TA}+\alpha \widetilde{D^TD}+\alpha I
%%\right)^{-1}(A^Tu^k+D^Tv^k+s^k)\\
%%u^{k+1}&=\prox_{\alpha g_1^*}(u^k+\alpha A(2x^{k+1}-x^k))\\
%%v^{k+1}&=\prox_{\alpha g_2^*}(v^k+\alpha D(2x^{k+1}-x^k))\\
%%s^{k+1}&=\prox_{\alpha g_3^*}(s^k+\alpha (2x^{k+1}-x^k))
%%\end{align*}
%%\begin{align*}
%%x^{k+1}&= x^k-\left(
%%\beta I+\alpha \widetilde{A^TA}+\alpha \widetilde{D^TD}+\alpha I
%%\right)^{-1}(A^Tu^k+D^Tv^k+s^k)\\
%%u^{k+1}&=S(u^k+\alpha A(2x^{k+1}-x^k);\alpha b_i)\\
%%v^{k+1}&=\Pi_{[-\lambda,\lambda]}(v^k+\alpha D(2x^{k+1}-x^k))\\
%%s^{k+1}&=(s^k+\alpha (2x^{k+1}-x^k))_+
%%\end{align*}

%eigenvalues are taken from
%\verb|https://en.wikipedia.org/wiki/Eigenvalues_and_eigenvectors_of_the_second_derivative|
%

\begin{figure}
    \centering
    \begin{subfigure}[b]{0.45\textwidth}
        \includegraphics[width=\textwidth]{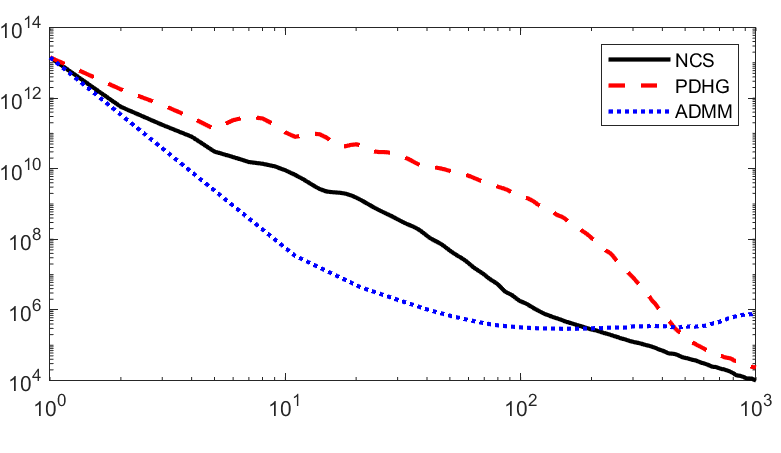}
        \caption{Parallel beam CT}
    \end{subfigure}
    ~
    \begin{subfigure}[b]{0.45\textwidth}
        \includegraphics[width=\textwidth]{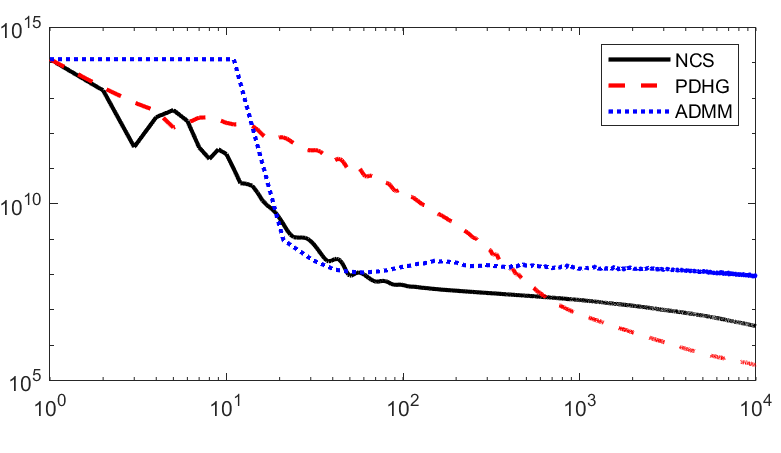}
        \caption{Fan beam CT}
\end{subfigure}
\\
    \begin{subfigure}[b]{0.45\textwidth}
        \includegraphics[width=\textwidth]{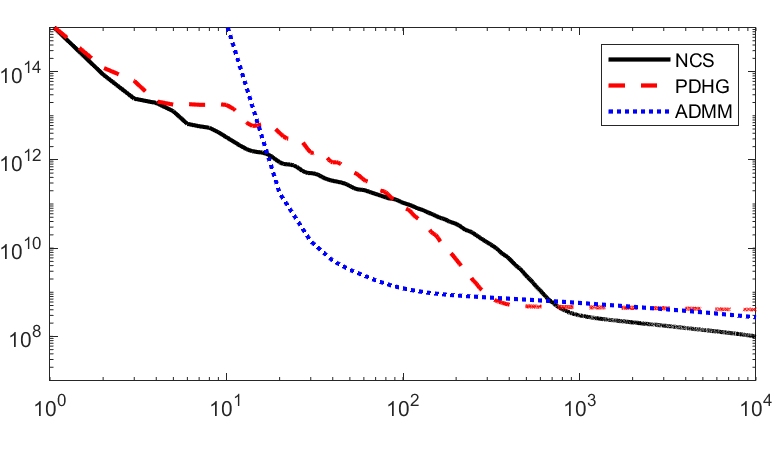}
        \caption{Cone beam (3D) CT}
    \end{subfigure}
    ~
    \begin{subfigure}[b]{0.45\textwidth}
        \includegraphics[width=\textwidth]{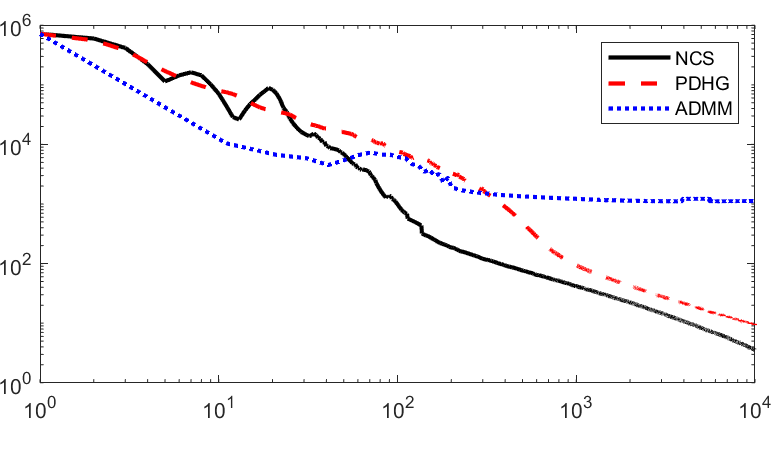}
        \caption{PET}
            \end{subfigure}
        \caption{Objective value suboptimality vs.\ iteration count. For ADMM, we count the number of inner-loop CG iterations.}
        \label{fig:convergence_plot}
\end{figure}

\begin{figure}
    \begin{subfigure}[b]{0.3\textwidth}
        \includegraphics[width=\textwidth]{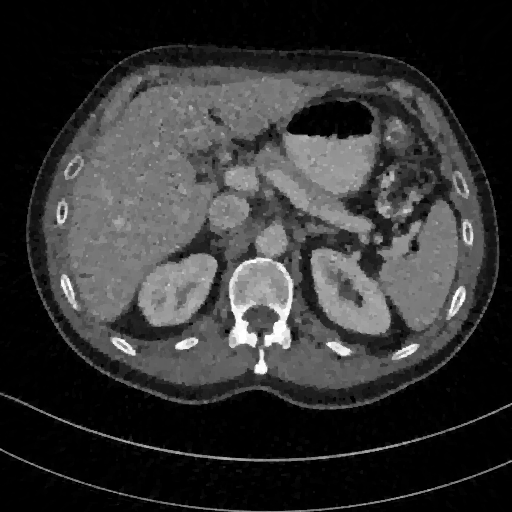}
        \caption{Parallel beam NCS}
            \end{subfigure}
    ~
    \centering
    \begin{subfigure}[b]{0.3\textwidth}
        \includegraphics[width=\textwidth]{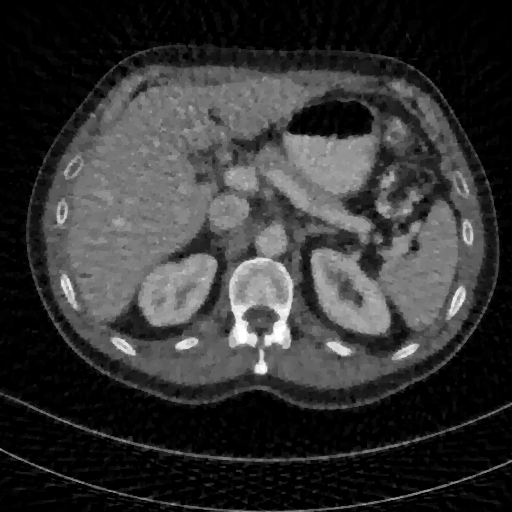}
        \caption{Parallel beam PDHG}
    \end{subfigure}
    ~
    \begin{subfigure}[b]{0.3\textwidth}
        \includegraphics[width=\textwidth]{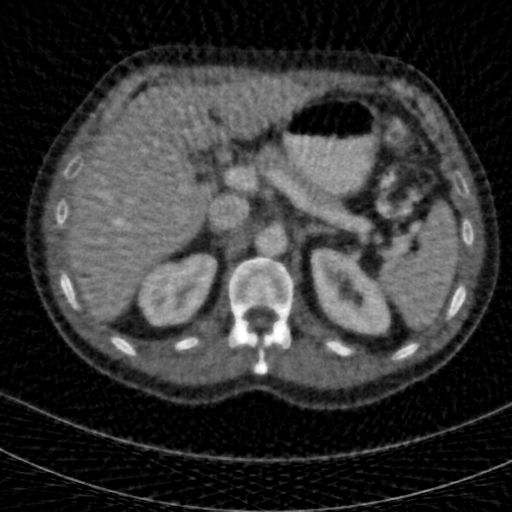}
        \caption{Parallel beam ADMM}
    \end{subfigure}\\
    \begin{subfigure}[b]{0.3\textwidth}
        \includegraphics[width=\textwidth]{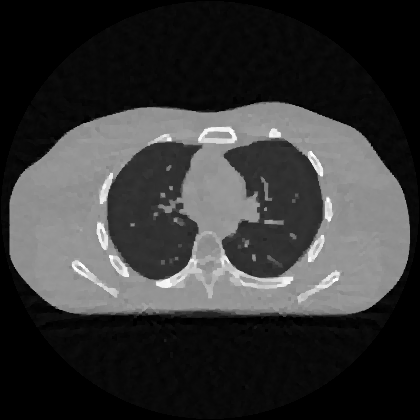}
        \caption{Fan beam NCS}
            \end{subfigure}
    ~
    \begin{subfigure}[b]{0.3\textwidth}
        \includegraphics[width=\textwidth]{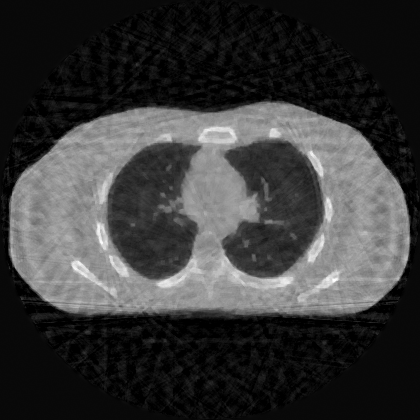}
        \caption{Fan beam PDHG}
    \end{subfigure}
    ~
    \begin{subfigure}[b]{0.3\textwidth}
        \includegraphics[width=\textwidth]{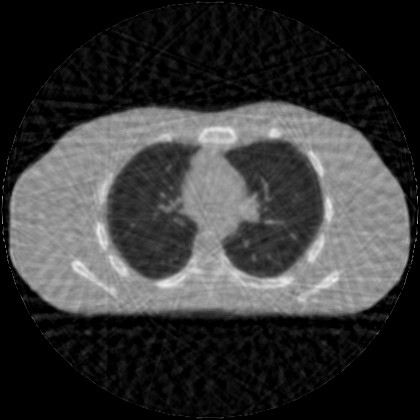}
        \caption{Fan beam ADMM}
    \end{subfigure}\\
    \begin{subfigure}[b]{0.3\textwidth}
        \includegraphics[width=\textwidth]{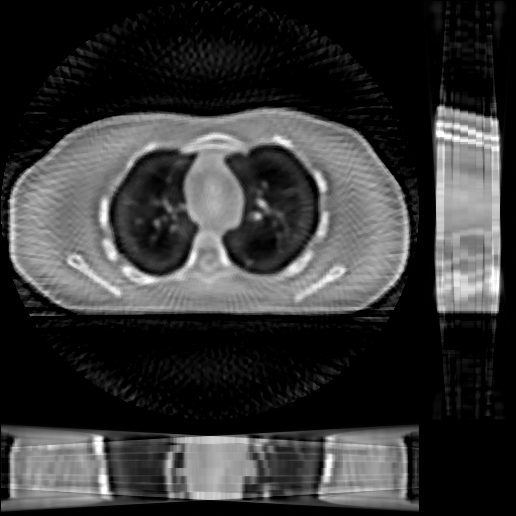}
        \caption{Cone beam (3D) NCS}
            \end{subfigure}
    ~
    \begin{subfigure}[b]{0.3\textwidth}
        \includegraphics[width=\textwidth]{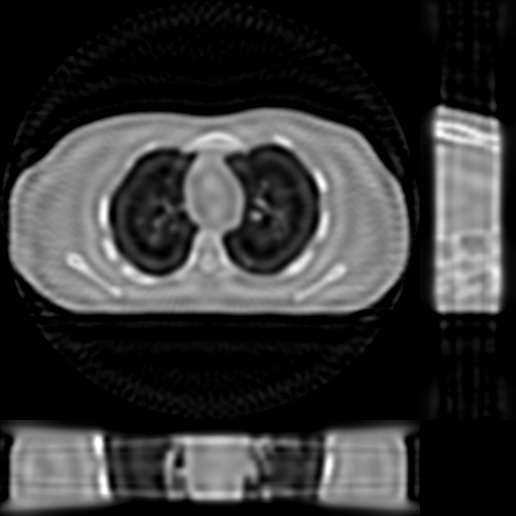}
        \caption{Cone beam (3D) PDHG}
    \end{subfigure}
    ~
    \begin{subfigure}[b]{0.3\textwidth}
        \includegraphics[width=\textwidth]{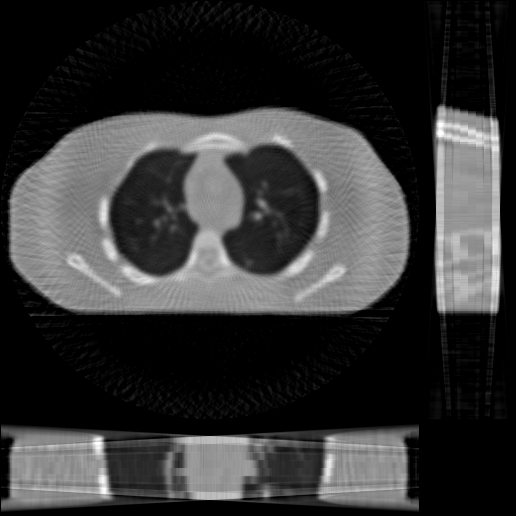}
        \caption{Cone beam (3D) ADMM}
    \end{subfigure}\\
    \begin{subfigure}[b]{0.3\textwidth}
        \includegraphics[width=\textwidth]{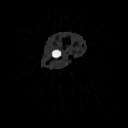}
        \caption{PET NCS}
            \end{subfigure}
    ~
    \begin{subfigure}[b]{0.3\textwidth}
        \includegraphics[width=\textwidth]{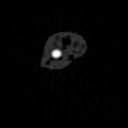}
        \caption{PET PDHG}
    \end{subfigure}
    ~
    \begin{subfigure}[b]{0.3\textwidth}
        \includegraphics[width=\textwidth]{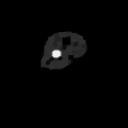}
        \caption{PET ADMM}
    \end{subfigure}
    \caption{
Parallel beam CT: $100$ iterations, runtime 2.38s, 2.21s, and 2.07s.
Fan beam:  $300$ iterations, runtime 4.31s, 4.41s, and 4.96s.
Cone beam:  $30$ iterations, runtime 54.29s, 50.33s, and 49.73s.
PET: $1000$ iterations, runtime 6.62s, 6.29s, and 5.47s.
For ADMM, we count the number of inner-loop CG iterations.}
\label{fig:CT}
\end{figure}

\section{Experiments}
\label{s:exp}
In this section, we present numerical experiments of NCS on a CUDA GPU.
%Overall, NCS provides a speedup over PDHG, and the GPU implementation provides an addition
%Furthermore, the GPU implementation provides an additional 10 to 50-fold speedup over the CPU implementation.
We compare the performance of NCS against PDHG and Ramani and Fessler's ADMM which uses conjugate gradient (CG) to solve the linear systems \cite{Ramani2012}.
Code is available at \url{https://github.com/kose-y/near-circulant-splitting}

We observe that the matrix-vector multiplication with $A$ and $A^T$ dominates the computational cost; the cost of the FFT in NCS and other operations is much smaller in comparison.
Consequently, the costs of one iteration of NCS, one iteration of PDHG, and one CG iteration for ADMM are comparable.
%The convergence plots of Figure~\ref{fig:convergence_plot} show NCS requires about $10$ times fewer iterations compared to PDHG to achieve an equivalent loss.
The convergence plots of Figure~\ref{fig:convergence_plot} show that NCS provides a speedup over PDHG and ADMM.
%This observation empirically validates Theorem~\ref{thm:main-rate}, and this translates to a $10$-fold speedup in wall-clock time.
%XXX comparison with ADMM XXX
Figure~\ref{fig:CT} shows reconstructed images.
Table~\ref{table:gpu-acceleration} shows the speedup of the GPU implementation over the CPU implementation.
Time measurements were taken on a system with an Intel Core i7-990X CPU running at 3.47GHz and a Titan Xp GPU.

%Figures~\ref{fig:CT} and \ref{fig:PET} respectively show the CT and PET imaging experiments.
%Figure~\ref{fig:convergence_plot} shows the convergence of PDHG and NCS in function value for the two experiments.

%For the fan beam and cone beam CT experiments, we follow the setup of \cite{zheng2018}.
%For the PET experiment, we follow the setup of \cite{Lim_2018}.
%We used synthetic extended cardiac-torso (XCAT) data \cite{segarsa2008} and real patient data provided by the Mayo Clinic XXX.

For parallel beam CT, we use real patient data provided by the Mayo Clinic (referenced in the acknowledgements).
For fan beam and cone beam CT, we use the synthetic extended cardiac-torso (XCAT) data \cite{segarsa2008} and follow the setup of \cite{zheng2018}.
For PET, we use the data of \cite{Lim_2018} and follow its setup.

%Parameter sensitivity
For all experiments, we roughly tuned the parameters of NCS and the other method at a resolution of  powers of $3$; i.e., we tested parameters $1\times 10^p$ and $3\times 10^p$ for $p=0,\pm1,\pm2,\dots$.
To get good performance, tuning the parameters well was important, but it was not a very difficult or sensitive process.

Parallel beam CT reconstructs a $512\times 512$ image with $\lambda = 10^{0}$.
The discrete Radon transform computed $60$ equiangular projections for a total of  $ 729\times 60=43740$ measurements.
The matrix-vector multiplication with respect to $R$ and $R^T$ were performed with the MATLAB GPU-accelerated \verb|radon| and \verb|iradon| functions.
For NCS, we used $\alpha=10^{-2}$, $\beta = 10^{-2}$, $\gamma=10^0$, and $H_{1,1}=10^{-1}$.
For PDHG, we used $\alpha=10^{-2}$, $\beta = 3\times 10^{-2}$, and $\gamma=10^1$.
For ADMM, we used $\alpha=10^{0}$, $\beta = 3\times 10^{-3}$, and $10$ CG inner iterations per outer loop.

Fan beam CT reconstructs a $420\times 420$ image with a sinogram of $11,100$ measurements and $\lambda = 10^1$.
%The discrete Radon transform computed $18$ equiangular projections for a total of  $1453\times 18=26154$ measurements.
The matrix-vector multiplications with respect to $E$ and $E^T$ were performed with the MATLAB GPU-accelerated sparse matrices.
For NCS, we used $\alpha=3\times10^{-3}$, $\beta = 10^{-2}$, and $\gamma=10^0$.
For PDHG, we used $\alpha=10^{-3}$, $\beta = 3\times 10^{-3}$, and $\gamma=10^2$.
For ADMM, we used $\alpha=10^{-4}$, $\beta = 3\times 10^{-2}$, and $10$ CG inner iterations per outer loop.

Cone beam CT reconstructs a 3D $420\times420\times 96$ image with a sinogram of $110,208$ measurements and $\lambda = 10^0$.
The matrix-vector multiplications with respect to $E$ and $E^T$ were performed with Fessler's MIRT toolbox \cite{mirt}.
For NCS, we used $\alpha=10^{-3}$, $\beta =10^{-1}$, and $\gamma=3\times10^{-2}$.
For PDHG, we used $\alpha=10^{-3}$, $\beta =10^{-1}$, and $\gamma=10^0$.
For ADMM, we used $\alpha=10^0$, $\beta = 10^{-1}$, and $10$ CG inner iterations per outer loop.

PET reconstructs a $128\times 128$ image with the geometry of $128$ equally spaced detectors on the unit circle and $\lambda = 10^{-3}$.
The matrix-vector multiplications with respect to $E$ and $E^T$ were performed with MATLAB GPU-accelerated sparse matrices.
For NCS, we used $\alpha= 10^{-3}$, $\beta = 10^{-3}$, $\gamma=10^{-4}$, and $H_{1,1}=10^{-2}$.
For PDHG, we used $\alpha=10^{-2}$, $\beta = 10^{-2}$, and $\gamma=3\times 10^{-2}$.
For ADMM, we used $\alpha=3\times 10^{-4}$, $\beta = 3\times 10^{-1}$, and $10$ CG inner iterations per outer loop.

For parallel beam CT and PET, where we use circulant approximations based on the analytic formula \eqref{eq:circ_R}, NCS consistently outperforms PDHG.
For fan beam and cone beam CT, where we use circulant approximations based on the more ad hoc formula \eqref{eq:H-fan}, NCS outperforms PDHG in specific regimes (in early iterations or later iterations).
ADMM with CG performs well in early iterations but worse in later iterations.

\begin{table}
\begin{center}
\begin{tabular}{ |c|c | c| c|}
  \hline	
 &
 \begin{tabular}{@{}c@{}}Intel Core i7-990X\\@ 3.47GHz\end{tabular}
  & TITAN Xp &Speedup\\
   \hline	
  Par Beam CT ($128\times128)$& $2.41 \mathrm{s}$ & $3.86\mathrm{s}$ & $0.62$x\\
    Par Beam CT ($256\times 256)$& $8.51\mathrm{s}$ & $4.49\mathrm{s}$ & $1.90$x\\
    Par Beam CT ($512\times 512)$& $38.92\mathrm{s}$ & $5.32\mathrm{s}$ & $7.32$x\\
    Par Beam CT ($1024\times1024)$& $198.53\mathrm{s}$ & $14.99\mathrm{s}$ & $13.2$x\\
   \hline	
  PET $(128\times 128)$&   $27.57\mathrm{s}$ &$4.07\mathrm{s}$ & $6.77$x\\
  PET $(256\times 256)$&   $109.68\mathrm{s}$ &$5.09\mathrm{s}$ & $21.5$x\\
  PET ($512\times 512)$&   $452.8\mathrm{s}$ &$9.39\mathrm{s}$ & $48.2$x\\
    \hline	
\end{tabular}
\end{center}
\caption{Runtimes of $1000$ NCS iterations with CPU and CUDA GPU implementations. 
The GPU implementation provides a significant speedup.
}
\label{table:gpu-acceleration}
\end{table}

\section{Conclusion}
In this paper, we presented NCS, a splitting method that leverages the near-circulant structure present in certain imaging applications.
For the problems we consider, ADMM/DRS requires an exact pseudoinverse, which is a significant computational burden. This requirement is relaxed for NCS; we only need to compute the pseudoinverse of an approximate linear system.
We theoretically analyzed NCS, and the result informs us of the effect and the advantage of using an approximate pseudoinverse. We presented medical imaging applications that exhibit near-circulant structures, which provide computationally efficient approximate pseudoinverses. We apply NCS to these problems, and, through experiments, empirically validate the theory and demonstrate that the algorithm can effectively utilize the parallel computing capability of a CUDA GPU.
The code used for our experiments is available at
\url{https://github.com/kose-y/near-circulant-splitting}

%
%Although this work only explored 2D imaging applications, NCS is more general. Investigating the effectiveness of NCS in 3D imaging problems, such as helical CT, is an interesting direction of future work.
% This approach should work with other constructions such as the fan beam construction so long as we have an analytical formula for the continuous counterpart.

%\appendix
\section{Appendix}
\label{s:appendix}
In this section, we quickly show how to obtain \eqref{eq:DRS/ADMM} from ADMM and DRS.
The two derivations are, in a sense, equivalent because
ADMM and DRS are, in a sense, equivalent methods \cite{gabay1983}.

\subsection{Derivation from ADMM}
Apply ADMM to \eqref{eq:RCP}
\begin{align*}
z^{k+1}&= \argmin_{z\in\mathbb{R}^m}\left\{
g(z)+\frac{\alpha}{2}\|z-Ax^k+b+\lambda^k\|^2
\right\}\\
x^{k+1}&=\argmin_{x\in \mathbb{R}^n}\left\{
\frac{\alpha}{2}\|z^{k+1}-Ax+b+\lambda^k\|^2
\right\}\\
\lambda^{k+1}&=\lambda^k+z^{k+1}-Ax^{k+1}+b.
\end{align*}
Rewrite this as
\begin{align*}
z^{k+1}&= 
\prox_{(1/\alpha)g}(Ax^k-b-\lambda^k)\\
x^{k+1}&=
(A^TA)^+ A^T(z^{k+1}+b+\lambda^k)\\
\lambda^{k+1}&=\lambda^k+z^{k+1}-Ax^{k+1}+b.
\end{align*}
Rewrite this using \eqref{eq:Moreau} as 
\begin{align*}
v^{k+1}&= \prox_{\alpha g^*}(\alpha (Ax^k-b-\lambda^k))\\
z^{k+1}&= 
Ax^k-b-\lambda^k-(1/\alpha)v^{k+1}\\
x^{k+1}&=
(A^TA)^+ A^T(z^{k+1}+b+\lambda^k)\\
\lambda^{k+1}&=A(x^k-x^{k+1})-(1/\alpha)v^{k+1}.
\end{align*}
Finally, we get
\begin{align*}
v^{k+1}&= \prox_{\alpha g^*}(v^k+\alpha (A(2x^k-x^{k-1})-b))\\
x^{k+1}&=
x^k-(1/\alpha)
(A^TA)^+ A^Tv^{k+1}.
\end{align*}

\subsection{Derivation from DRS}
Write the dual problem \eqref{eq:dual} as
\[
\begin{array}{ll}
\underset{u\in \mathbb{R}^m}{\mbox{minimize}} &g^*(u)+u^Tb
+\delta_{
\{
u\in \mathbb{R}^m\,|\,A^Tu=0
\}}(u).
\end{array}
\]
Apply DRS to this dual formulation to get
\begin{align*}
u^{k+1/2}&=
(I-A(A^TA)^+ A^T)z^k\\
u^{k+1}&=\prox_{\alpha g^*}(2u^{k+1/2}-z^k-\alpha b)\\
z^{k+1}&=z^k+u^{k+1}-u^{k+1/2}.
\end{align*}
Define $Ax^{k+1}$ with
\[
z^k=u^{k+1/2}-\alpha Ax^{k+1}.
\]
Then 
\begin{align*}
Ax^{k+1}&=-(1/\alpha)
A(A^TA)^+ A^Tz^k\\
u^{k+1}&=\prox_{\alpha g^*}(
z^k-2A(A^TA)^+ A^Tz^k
-\alpha b)\\
z^{k+1}&=u^{k+1}-\alpha Ax^{k+1}
\end{align*}
and
\begin{align*}
Ax^{k+1}&=A\left(x^k-(1/\alpha)
(A^TA)^+ A^Tu^k\right)\\
u^{k+1}&=\prox_{\alpha g^*}(
u^k+\alpha (A(2x^{k+1}-x^k)-b)).
\end{align*}
The algorithm depends on $x^k$ only through the product $Ax^k$. Therefore, we can write
\begin{align*}
x^{k+1}&=x^k-(1/\alpha)
(A^TA)^+ A^Tu^k\\
u^{k+1}&=\prox_{\alpha g^*}(
u^k+\alpha (A(2x^{k+1}-x^k)-b)).
\end{align*}

\section*{Acknowledgments}
We would like to thank Wotao Yin for his many helpful comments.
We gratefully acknowledge the support of NVIDIA Corporation with the donation of the Titan Xp GPU used for this research.
We thank Dr.\ Cynthia McCollough, the Mayo Clinic, the American Association of Physicists in Medicine, and grants EB017095 and EB017185 from the National Institute of Biomedical Imaging and Bioengineering for providing the real patient data used in the experiments.

\bibliographystyle{siamplain}
\bibliography{CT}
\end{document}